\documentclass[12pt,notitlepage]{amsart}
\usepackage{latexsym,amsfonts,amssymb,amsmath,amsthm} 
\newcommand{\N}{\ensuremath{\mathbb N}}
\newcommand{\Z}{\ensuremath{\mathbb Z}}
\newcommand{\R}{\ensuremath{\mathbb R}}

\newcommand{\Q}{\ensuremath{\mathbb Q}}
\newcommand{\C}{\ensuremath{\mathbb C}}

\newcommand{\xbar}{\overline{x}}
\newcommand{\ybar}{\overline{y}}

\theoremstyle{plain}		
	\newtheorem{theorem}{Theorem}[section]
	\newtheorem{prop}[theorem]{Proposition}
	\newtheorem{cor}[theorem]{Corollary}
     \newtheorem{lemma}[theorem]{Lemma}
	\newtheorem{definition}[theorem]{Definition}

\theoremstyle{remark}		
	
	\newtheorem*{remark}{Remark}

\numberwithin{equation}{section}
\begin{document}

\title{Beyond endoscopy for the Rankin-Selberg L-function}
\author{P. Edward Herman} 
\address{American Institute of Mathematics, 360 Portage Ave.,
Palo Alto, CA 94306-2244}
\email{peherman@aimath.org}
\thanks{This research was supported by an NSF Mathematical Sciences Research Institutes Post-Doc and by the American Institute of Mathematics.}
\begin{abstract}
We try to understand the poles of L-functions via taking a limit in a trace formula. This technique avoids endoscopic and Kim-Shahidi methods. In particular, we investigate the poles of the Rankin-Selberg L-function. Using analytic number theory techniques to take this limit, we essentially get a new proof of the analyticity of the Rankin-Selberg L-function at $s=1.$ Along the way we discover the convolution operation for Bessel transforms. \end{abstract}
\maketitle

\section{Introduction}

In this paper we present further calculations using Langlands' beyond
endoscopy idea. We roughly describe this concept here. Take a
 cuspidal holomorphic or Maass form
$\phi$ and an associated L-function $L(s)=\sum_{n=1}^\infty
\frac{b_n(\phi)}{n^s},$
 where $b_n(\phi)$ are associated complex parameters defined by the dual group. The focus is on $$\lim_{X \to \infty} \frac{1}{X}
 \sum_{n \leq X} b_n(\phi).$$
The limit is the residue of the pole of the associated L-function at
$s=1.$ As it is difficult to study just one form in this way, we
rather study this limit as we sum over all modular forms. This
allows us to use the trace formula. Summing then over the spectrum
$\phi$ along with the averaging of $b_n(\phi)$ will ``detect" the
associated L-functions that have poles.

    We provide motivation for our work with a summary of Langlands
    original idea in \cite{Lan04}.

\section{Langlands' Beyond Endoscopy}

        Let $\mathbb{A_Q}$ be the ring of adeles of $\Q,$ and $\pi$ be
    an automorphic cuspidal representation of
    $GL_2(\mathbb{A_Q}).$ We define $m(\pi,\rho)$ to be the order of the
    pole at $s=1$ of $L(s,\pi,\rho),$ where $\rho$ is a
    representation of the dual group $GL_2(\C).$

        Langlands proposes the study of \begin{equation} \label{p_sum} \lim_{X \to \infty} \sum_\pi \frac{1}{X}
        tr(\pi)(f)\sum_{p\leq X} \log (p) a(p,\pi,\rho).\end{equation} Here $f$
        is a nice test function on $GL_2(\mathbb{A_Q}),$ and
        $tr(\pi)(f)$ is the trace of the operator defined by $f$
        on $\pi.$ $a(p,\pi,\rho)$ is the $p-$th Dirichlet
        coefficient of $L(s,\pi,\rho).$ The quantity $$\lim_{X \rightarrow
        \infty}\frac{1}{X}\sum_{p\leq X} \log (p) a(p,\pi,\rho),$$
        is equal to $m(\pi,\rho).$

            Therefore, summing over the range of representations
            $\pi$  will "detect" the ones which have nontrivial multiplicity. The tool used to study this sum
            over the spectrum of forms $\pi$ is the trace formula.
            Ultimately, one gets from use of the trace formula a sum over
            primes and conjugacy classes, and hopes by analytic
            number theory techniques to take the limit. One hopes
            that after getting the limit, one can decipher and
            construct the L-functions having non-trivial
            multiplicity of the pole at $s=1.$ 
            Sarnak addresses (\ref{p_sum}) in \cite{Sar} for $\rho$
            the standard representation. He points out that such a
            computation can be done, but the tools used for the
            study of sums of primes is limited, and this problem
            is perhaps more tractable if rather studied over the
            sum of integers.

For the standard $L$-function the idea then is to evaluate \begin{equation} \label{nsum} \lim_{X
\to \infty} \sum_\pi \frac{1}{X}
        tr(\pi)(f)\sum_{n\leq X}  a(n,\pi,\rho).\end{equation}
        This should "detect," rather than the multiplicities of the
        poles, the residue of the poles of the associated
        L-functions. We do this because the trace formula with the
        easiest analytic application for $GL_2$ is the Kuznetsov trace
        formula, and the sum over integers compliments such a
        limit.

        Rather than use the adelic language, we use the classic
        Petersson-Kuznetsov trace formula. Then for the standard $L$-function, Sarnak \cite{Sar} showed, up to some weight factors needed in the trace formula,  $$ \sum_{n \leq X}g(n/X) \sum_f a_n(f)=O(X^{-A})$$ for any $A>0.$ This is equivalent to $L(s,f)=\sum_{n=1}^{\infty}\frac{a_n(f)}{n^s}$ being entire. Here $g\in C_0^{\infty}(\R^{+})$ and $a_n(f)$ are normalized Fourier coefficients of the cusp form $f,$ and the spectral sum ranges over an orthonormal
basis of holomorphic and Maass forms of a certain level and
nebentypus.
     Further work was done by Venkatesh (\cite{Venk1},\cite{Venk2}) for the symmetric square $L$-function. There the focus was taking the limit for $$\lim_{X \to \infty} \frac{1}{X} \sum_{n \leq X} g(n/X) \sum_f a_{n^2}(f).$$ He showed the symmetric square $L$-function of a cusp form has a pole if it is induced from a Hecke character over a quadratic field. 
     We go further and compute  $$\lim_{X \to \infty} \frac{1}{X} \sum_{n \leq X}  \sum_g \sum_f a_{n}(f)a_n (g).$$ This would be inspecting the poles of the Rankin-Selberg $L$-function. From the Rankin-Selberg
theory we expect some L-functions to have a pole at $s=1.$ Namely,
we expect the L-functions $L(s,\phi \times \bar{\phi})$, where
$\bar{\phi}$ is the modular form with conjugate Fourier
coefficients to $\phi$ to have a pole.

\section{Beyond endoscopy for the Rankin-Selberg L-function}
        Using a beyond endoscopic
approach, we want to show only those $L(s,\phi \times \bar{\phi})$ associated
to forms $\phi$ remain.

At the heart of our study is the limit of a product of Kuznetzov
trace formulas. Given a smooth function $V$ on $\R^{+}$ of compact
support, and positive integers $n, l$, we consider the Kuznetzov
trace formula
\begin{equation}\label{eq:kuzdef}
K_{n,l}(V): =S_{n,l}(V) +C_{n,l}(V),\end{equation}where $$S_{n,l}(V):=\sum_{\phi}
h(V,\lambda_{\phi})a_n(\phi)\overline{a_l(\phi)},$$
$$C_{n,l}(V):=\frac{1}{4\pi}
\int_{-\infty}^\infty h(V,t)\eta(n,1/2
+it)\overline{\eta(l,1/2+it)}dt.$$  
Here $h(V,\lambda)$ is a certain transform of $V$, and $a_n(\phi)$
are normalized Fourier coefficients of a form $\phi,$ which is
either a holomorphic or Maass form. The term $\eta(l,1/2+it)$ is a normalized divisor function as in \cite{Iw}. These are normalized Fourier coefficients of the Eisenstein series. We will define the technical
details of this sum in more detail in section 4.

Suppose that $W$ is a second function of the same type as $V$, and
let $g\in C_0^{\infty}(\R^{+})$ be a function satisfying
$\int_0^\infty g(t)dt=1$.

We shall study the following limit:
\begin{equation}\tag{L}\lim_{X \to \infty} \frac{1}{X} \sum_{n \in \Z}
g(n/X) K_{n,l}(V)  K_{n,l'}(W).
\end{equation}
 Define
\begin{equation} V*W(z):=  \int_{-\infty}^{\infty}
\int_{-\infty}^{\infty}
  \exp \left(z\frac{i}{2}(
 \frac{x}{y}+\frac{y}{x})\right) \exp \left( (\frac{1}{z})\frac{8\pi^2i}{xy}\right)\times \end{equation} $$V(\frac{4\pi
}{x})  W(\frac{4\pi}{ y})\frac{dx}{x} \frac{dy}{y},
$$
then we prove that $V*W$ is the convolution operation for Bessel transforms. That is, $\lambda_{\phi}$ is the archimedean
parameter associated to a form $\phi,$ and
\begin{equation}
 h(V,\lambda)
 := \left\{ \begin{array}{ll}
         i^{k}\int_0^\infty V(x)J_{\lambda -1}(x)x^{-1}dx & \text{if } \lambda  \in  2 \Z; \medskip \\
         \int_0^\infty
V(x)B_{2i\lambda}(x)x^{-1}dx & \text{if } \lambda \in
\R-2\Z.\end{array} \right.  \end{equation} Here, $B_{2it}(x) = (
\text{2 sin}(\pi it))^{-1}(J_{-2it}(x) - J_{2it}(x)),$ where
$J_\mu(x)$ is the standard $J$-Bessel function of index $\mu$ (See
\cite{IK} and \cite{Wat}).We call it the $B$-Bessel function.
\begin{theorem} \label{main theorem} For all $V, W$ as above, $h(V*W,t)=C_{t} h(V,t)h(W,t),$ where $C_{t}=2\pi$ for $t$ an even integer, and $C_{t}=\pi$ for $t$ purely imaginary.
\end{theorem}
  Such results are valuable for inverting test functions in the trace formula, and are highly sought after for higher rank trace formulae. This beyond endoscopic approach could possibly help.

 The main theorem proved in the paper is
\begin{theorem}\label{main theorem1} Let $l,l'$ be positive integers, then 
  $$\lim_{X \to \infty} \frac{1}{X} \sum_{n}
g(n/X) K_{n,l}(V)  K_{n,l'}(W) =  \frac{12}{\pi}
K_{l,l'}(V*W)$$

\end{theorem}

 We find it
extremely interesting that if one looks at Theorem \ref{main
theorem1} strictly from the geometric sides of the trace formula,
one has
\begin{cor}\label{thm:lowerbound} \begin{equation*}  \lim_{X \to \infty} \frac{1}{X} \sum_{n
\in \Z}g(n/X) \left( \sum_{c_1=1}^\infty \frac{1}{c_1}
S(l,n,c_1)V(4\pi\sqrt{nl}/c_1)\right) \left(\sum_{c_2=1}^\infty
\frac{1}{c_2} S(l',n,c_2)W(4\pi\sqrt{nl'}/c_2)\right)=\end{equation*}
$$\sum_{d=1}^\infty \frac{1}{d}
S(l,l',d)(V*W)(4\pi\sqrt{ll'}/d),$$ where $$
S(r,s,d):=\sum_{x \in (\Z/ d\Z)*} e(\frac{r \overline{x}+s
x}{c})$$ is the Kloosterman sum.
\end{cor} The average of
a product of sums of Kloosterman sums is another sum of
Kloosterman sums. There is perhaps further application in this
statement.

We also prove the cuspidal (resp. continuous) parts of the limit $(L)$ match with themselves, and the cuspidal and continuous parts are orthogonal.  

\begin{theorem}\label{cupcon1}
$\lim_{X \to \infty} \frac{1}{X} \sum_{n}
g(n/X) S_{n,l}(V)  S_{n,l'}(W)=\frac{12}{\pi}S_{l,l'}(V*W).$
\end{theorem}

\begin{theorem}\label{cupcon2}
$\lim_{X \to \infty} \frac{1}{X} \sum_{n}
g(n/X) C_{n,l}(V)  C_{n,l'}(W)=\frac{12}{\pi}C_{l,l'}(V*W).$
\end{theorem}

\begin{theorem}\label{cupcon3}
$ \lim_{X \to \infty} \frac{1}{X} \sum_{n}
g(n/X) S_{n,l}(V)  C_{n,l'}(W)=0.$

\end{theorem}

In a forthcoming paper we plan to add Hecke operators into our trace formula to obtain the analytic continuation of the Rankin-Selberg L-function.

Very generally these theorems say if we apply the trace formula to the
spectral sums $\phi,\psi$ in $(L)$ to get the geometric sides of
the formula, and take the limit as $X\rightarrow \infty, (L)$ is
equal to just a single spectral sum.

{\bf Acknowledgements.} The author would like to thank his advisor Jonathan Rogawski for proposing this problem as part of his thesis, as well as his very helpful ideas. The author also would like to acknowledge the useful conversations with Brian Conrey, Eric Ryckman, and Peter Sarnak. Finally, the author appreciates all the helpful points from the referee.

\section{Preliminaries}

        We start by defining the Kuznetsov trace formula used in this
    Chapter and its normalization. We refer to  \cite{Iw} book on
    it's derivation. Let $S(\Gamma_0(N))$ be the space of
    holomorphic cusp forms of weight $k$ for the group $\Gamma_0(N).$ For each
    form $\phi \in S_k(\Gamma_0(N)),$ let $c_n(\phi)$ be the $n$-th
    Fourier coefficient, then define $$a_n(\phi):=
    \sqrt{\frac{\pi^{-k} \Gamma(k)} {(4n)^{k-1}}}c_n(\phi).$$
Likewise, for Maass cusp forms we define
        $$a_n(\phi):=(\frac{4\pi |n|}{\cosh(\pi
        s)})^{1/2}\rho(n),$$ where $\phi$ has $L^2$ norm one and
        eigenvalue $1/4+s^2$ with Fourier expansion $$
        \phi(z)=\sum_{n \neq 0} \rho(n)W_s(nz).$$ Here $W_s(nz)=2\sqrt{y}
        K_{s-1/2}(ny)e(x).$
 The continuous spectrum coefficients are defined as $$\eta(l,1/2+it):=2\pi^{1+it}{\cosh(\pi t)}^{-1/2} \frac{\tau_{it}(n)}{\Gamma(1/2+it)\zeta(1+2it)},$$ 
where $\tau_{it}(n)=\sum_{ab=n}(a/b)^{it}.$

    The Kuznetsov formula states
\begin{equation}\label{eq:kuznet}\sum_{\phi}
h(V,\lambda_{\phi})a_n(\phi)\overline{a_l(\phi)}+\frac{1}{4\pi}
\int_{-\infty}^\infty h(V,t)\eta(n,1/2
+it)\overline{\eta(l,1/2+it)}dt\end{equation} $$=
\sum_{c=1}^\infty \frac{1}{c} S(l,n,c)V(4\pi \sqrt{ln}/c)$$
 where the sum $\phi$ is
over an orthonormal basis for $S_k(\Gamma_0), k \in 2 \Z$ and
Maass forms w.r.t. the Petersson inner product, and $V \in
C_0^{\infty}(\R-\{0\}).$

\section{Expectation of poles of Rankin-Selberg L-function}

    We focus on the holomorphic forms, the Maass forms are analogous.
Classically, the Rankin-Selberg L-function is defined as
$$L(s,\phi \times \psi)=\zeta(2s) \sum_{n=1}^\infty c_n(\phi)
d_n(\psi)n^{-s},$$ for cuspidal Hecke eigenforms $\phi,\psi \in
S_k(\Gamma_0),$ with Fourier (unnormalized) coefficients
$c_n(\phi), d_n(\psi),$ respectively. The work of Rankin and
Selberg show this L-function has much of the same good analytic
properties of Hecke and automorphic L-functions: analytic
continuation, a functional equation, and an Euler product. They
show further if $\phi=\bar{\psi},$ then $$ \text{Res}_{s=k}
\sum_{n=1}^\infty c_n(\phi) d_n(\psi)n^{-s}=\frac{3}{\pi}
\frac{(4\pi)^k}{\Gamma(k)}<\phi,\phi>,$$ and the L-function is
entire else. Now taking into consideration the normalization from
the previous section, this same statement about the poles of the
Rankin-Selberg L-function is
$$ \langle \phi,\phi\rangle^{-1} \text{Res}_{s=1} \sum_{n=1}^\infty  a_n(\phi)
b_n(\psi)n^{-s}=\frac{12}{\pi}.$$

    We ask now what do we expect from a beyond endoscopy
    calculation for a product of two Kuznetsov formulas. Now assuming the basis of automorphic forms is orthonormal, one studies \begin{equation}\label{eq:nsum} \lim_{X \to \infty} \frac{1}{X} \sum_{n \in \Z}g(n/X)
\left(\sum_{\phi} h(V,\lambda_{\phi})a_n(\phi)\overline{a_l(\phi)}
+ \{C.S.C.\}_{n,l}\right)\end{equation}
$$\left(\sum_{\psi}
h(W,\lambda_{\psi})b_n(\phi)\overline{b_{l'}(\phi)} + \{C.S.C.\}_{n,l'}\right),$$ where $\{C.S.C.\}_{i,j}$ stands for the continuous spectrum contribution with Fourier coefficient parameters $i,j$ as in \eqref{eq:kuznet}.  If we are free to interchange sums and
limits, the heart of the calculation boils down to investigating
the smooth sum over $n,$ \begin{equation}\label{eq:ini}
\frac{1}{X} \sum_{n \in
\Z}g(n/X)a_n(\phi)b_n(\psi).\end{equation} Via
Mellin inversion, \eqref{eq:ini} equals 
\begin{equation}\label{eq:ini1}
\frac{1}{2\pi i} \int_{\sigma -i\infty}^{\sigma+i\infty} G(s)\frac{L(s,\phi \times \psi)}{\zeta(2s)}X^s ds, 
\end{equation}
where $G(s)=\int_0^\infty g(x)x^{s-1}dx$ is the Mellin transform, and $\sigma >2$ to ensure the convergence of the integral. Assuming Rankin-Selberg theory, we make a contour shift to $\sigma_1=1-\epsilon,$ with $\epsilon>0$ sufficiently small. Then \eqref{eq:ini1} equals $$\frac{12\delta_{\phi,\psi}}{\pi} + O(X^{-\epsilon}),$$ where  \begin{equation*} \delta_{\phi,\psi}
 := \left\{ \begin{array}{ll}
         1 & \text{if } \phi=\bar{\psi},  \medskip \\
       0 & \text{if } \phi\neq \bar{\psi}. \end{array} \right. \end{equation*}

 Therefore, we expect this sum over $n$ to be non-trivial when
$\phi=\bar{\psi}$ with residue $\frac{12}{\pi},$ or analogously \eqref{eq:nsum} equals $$\frac{12}{\pi}\sum_{\phi} h(V,\lambda_{\phi})h(W,\lambda_{\phi})a_l(\phi)b_{l'}(\phi) + \{C.S.C.\}_{l,l'}.$$

 This is precisely
the statement of Theorem \ref{main theorem1}. The problem is we
cannot freely interchange the spectral sum and the limit in
\eqref{eq:nsum}. However, after using the Kuznetsov trace formula
for both spectral sums and some analysis we can take this limit.

\begin{remark}
Theorem \ref{main theorem1} can be proved using Voronoi summation,
very similar to Chapter 3 in \cite{Venk1}. The author focuses on
using two Kuznetsov formulas instead, because the Voronoi
summation argument does not work in the Asai L-function case which was the focus of the author's thesis. The author has notes proving the result using
Voronoi summation as well, but chose not to incorporate them into the
paper.
\end{remark}

\section{Number-theoretic lemmas}
We prove some number-theoretic lemmas that are crucial to our
calculation. Using standard terminology, let $x$ be a representative class modulo $c$ such that $(x,c)=1.$ We then denote $\overline{x} \mod c$ as the element such that $\overline{x}x\equiv 1(c).$ 

\begin{definition}\label{equiv}
Let $\overline{X}(c_1,c_2,n)$ denote the equivalence classes of
pairs $(x,y)$ with $x,y\in\mathbf{Z}$ such that $(x,c_1)=1$,
$(y,c_2)=1$, and
$$
c_2x+c_1y = n.
$$
Here we say that $(x,y)$ is equivalent to $(x',y')$ if $x\equiv
x'\pmod{c_1}$ and $y\equiv y'\pmod{c_2}$. Let $X(c_1,c_2,n)$ be a
set of representatives for the classes in
$\overline{X}(c_1,c_2,n)$.
\end{definition}

\begin{prop}\label{lem:n00}
Let $(x,y)\in X(c_1,c_2,0),$ then $x=-y, c_1=c_2.$
\end{prop}
\begin{proof}
 It is sufficient to study
$$c_2x \equiv 0 (c_1). $$ Since $(x,c_1)=1,$ we
have $c_2=\overline{x}\gamma c_1,\gamma \in \Z.$ Likewise, $$c_1y
\equiv 0 (c_2). $$  implies $c_1=\overline{y}\gamma' c_2.$ This
implies $c_1=c_2.$ Certainly then $$ c_1(x+y)= 0,$$ or $x=-y.$
\end{proof}

It is assumed, unless stated otherwise, $n \neq 0.$

\begin{prop}\label{inverse} Let $(x,y)\in X(c_1,c_2,n)$ and $\xbar \in\mathbf{Z} $ be an inverse of $x$ modulo $c_1$ and $\ybar \in\mathbf{Z}$ be an inverse of $y$ modulo
$c_2$. Then there exists a pair $(r_1,r_2)$ such that
$r_1r_2\equiv 1\pmod{n}$ and
\begin{equation}\label{a}
\xbar = \frac{c_2+c_1r_1}{n},\qquad \ybar = \frac{c_1+c_2r_2}{n}
\end{equation}
The pair $(r_1,r_2)$ is uniquely determined modulo $n$ by the
equivalence class of the pair $(x,y)$, and the map from
$X(c_1,c_2,n)$ to the set of pairs $(r_1,r_2)$ modulo $n$ is
injective.
\end{prop}

\begin{proof} Set
$$
r_1 = \frac{n\xbar-c_2}{c_1},\qquad r_2 = \frac{n\ybar-c_1}{c_2}
$$
Note that $r_1$ is an integer because
$$
n\xbar-c_2 = (c_2x+c_1y)\xbar - c_2 = c_2(x\xbar -
1)+c_1y\xbar\equiv 0 \pmod{c_1}
$$
Similarly, $r_2$ is an integer.

It is clear that $(r_1,r_2)$ is determined by the pair
$(\xbar,\ybar)$.  If we replace $\xbar$ by $\xbar' = \xbar+\mu
c_1$,  $r_1$ is replaced by
$$
r_1' =  r_1+\mu n
$$
Therefore $r'_1\equiv r_1\pmod{n}$ as claimed. Similarly, $(x,y)$
determines $r_2$  modulo $n$.

If two pairs $(x,y)$ and $(x',y')$   in $X(c_1,c_2,n)$ are both
associated to $(r_1,r_2)$, then $\xbar=\overline{x'}$ and $\ybar =
\overline{y'}$. Therefore $x\equiv x'\pmod{c_1}$ and $y\equiv
y'\pmod{c_2}$.

 Finally,
\begin{align*}
r_1r_2 & =  \left(\frac{n\xbar -
c_2}{c_1}\right)\left(\frac{n\ybar - c_1}{c_2}\right)  = 1 +
\frac{n^2\xbar\ybar-n\xbar c_1-n\ybar c_2}{c_1c_2} = 1 +
n\frac{n\xbar\ybar-\xbar c_1-\ybar c_2}{c_1c_2} \end{align*}
 But
 $$
 n\xbar\ybar = (c_2x+c_1y)\xbar\ybar = c_2x\xbar\ybar +  c_1\xbar y\ybar
 $$
 so we have
 \begin{align*}
r_1r_2 &=  1 + n\frac{c_2x\xbar\ybar +  c_1\xbar y\ybar-\xbar c_1-\ybar c_2}{c_1c_2}\\
&=  1 + n\left[\frac{c_2(x\xbar -1)\ybar + c_1(y\ybar - 1)\xbar}{c_1c_2}\right]\\
\end{align*}
The expression in brackets is an integer, so   $r_1r_2\equiv
1\pmod{n}$.
\end{proof}

 \begin{definition} Let $c_1$, $c_2$ be positive integers. Set  $d = (c_1,c_2)$. Assume that $d|n$.
 Let $Y(c_1,c_2,n)$ be the set of  classes $r\in(\Z/n)^*$ such that
\begin{enumerate}
\item[(a')] $(c_1/d)r+(c_2/d) \equiv 0\pmod{\frac{n}{d}}$
\item[(b')] $(c_1/d)r+(c_2/d) \not\equiv 0 \pmod{\frac{n}{d'}}$
if  $d'|d$ and  $d'<d$.
\end{enumerate}
 \end{definition}

 \begin{prop}\label{bijection} The map $i:(x,y)\to r_1$ defines a bijection between $X(c_1,c_2,n)$
 and $Y(c_1, c_2,n)$.
 \end{prop}
 \begin{proof} Let $(x,y)\in X(c_1,c_2,n)$. We show that the associated $r_1$ belongs to $Y(c_1,c_2,n)$.
Then
$$
\xbar = \frac{ c_1r_1+c_2}{n} = \frac{(c_1/d)r_1+(c_2/d)}{(n/d)}
$$
Therefore, $\frac{c_1r_1}{d}+\frac{c_2}{d}\equiv 0\pmod{\frac{n}{d}}$ and (a') is
satisfied.  Suppose that $m$ is a proper divisor of  $d$
 and let $d'=d/m$. We claim that  $\frac{c_1}dr+\frac{c_2}d\not\equiv 0\pmod{\frac{n}{d'}}.$ If this were not the case,
we would have
$$
\xbar = \frac{(c_2/d)+(c_1/d)r_1}{(n/d)} =
m\frac{(c_2/d)+(c_1/d)r_1}{(n/d')}
$$
This would imply that $m$ divides $\xbar$, which contradicts the
fact that $\xbar$ is a unit modulo $c_1$. Therefore (b') is
satisfied and $r\in Y(c_1,c_2,n)$. Furthermore, the map $i$ is
injective on $X(c_1,c_2,n)$  by Proposition \ref{inverse}. Next,
assume that $Y(c_1,c_2,n)$ is non-empty.
 Let $r\in\Z$ be prime to $n$ and assume that $r\pmod{n}$ belongs to $Y(c_1,c_2,n)$.
Set
$$
\xi = \frac{(c_1/d)r+(c_2/d)}{n/d} = \frac{c_2 +  c_1r}n
$$
Then $\xi$ is relatively prime to $d$ because
$(c_1/d)r+(c_2/d)\not\equiv 0\pmod{n/d'}$ for all proper
divisors $d'$ of $d$. On the other hand, if $q$ is a common factor
of both $\xi$ and  $c_1/d$, then  $q|c_2/d$. But $(c_1/d,c_2/d)=1$
so $q=1$. This proves that $\xi$ is prime to both $d$ and $c_1/d$,
and hence is a unit modulo $c_1$. Now choose $x\in\Z$ such that
$x\xi\equiv 1\pmod{c_1}$ and set $\xbar = \xi$. Then
$$
x\xbar = 1 + \mu c_1
$$
for some $\mu\in\Z$. We claim that there exists $y\in\Z$ such that
$$
c_2 x +  c_1y = n
$$
In fact,
$$
y = \frac{n-c_2x}{c_1}
$$
To show that $y\in \Z$, observe that $c_2 = \xbar n -c_1r$ and so
$$
y = \frac{\left(n - (\xbar n-c_1r)x\right)}{c_1} =
\frac{\left(n(1-x\xbar) +  c_1r_1x\right)}{c_1} =   r_1x  -  n\mu
$$
Thus we have produced a pair $(x,y)\in X(c_1,c_2,n)$ that maps to
$r \pmod{n}$. This proves the surjectivity.
\end{proof}

\section{Rewriting the Geometric Side}
By the Kuznetsov trace formula, the limit (L) is equal to
\begin{equation} (L)  = \lim_{X \to \infty} \frac{1}{X} \sum_{n \in \Z}g(n/X)
\left( \sum_{c_1=1}^\infty \frac{1}{c_1}
S(l,n,c_1)V(4\pi\sqrt{nl}/c_1)\right)\times \end{equation}
$$\left(\sum_{c_2=1}^\infty \frac{1}{c_2}
S(l',n,c_2)W(4\pi\sqrt{nl'}/c_2)\right)$$ We first reorganize the
terms.  \begin{equation}\lim_{X \to \infty} \frac{1}{X} \sum_{n}
g(n/X)  \sum_{c_1,c_2} \frac{1}{c_1 c_2}
S(l,n,c_1)S(l',n,c_2)V(\frac{4\pi \sqrt{nl}}{c_1})W(\frac{4\pi
\sqrt{nl'}}{c_2}).\end{equation} We can do this because the
$c_1,c_2,$ sums are finite.

We now break up the Kloosterman sums and gather all the $n$-terms.

\begin{equation}\label{sm} \lim_{X \to \infty}\frac{1}{X} \sum_{c_1,c_2} \frac{1}{c_1 c_2}\sum_{x  (c_1)^{*}}
 \sum_{y  (c_2)^{*}} e(\frac{\overline{x}l}{c_1}+\frac{\overline{y}l'}{c_2}) \end{equation}
 $$\left\{
 \sum_{n \in \Z}
 e(n(\frac{x}{c_1}+\frac{y}{c_2}))g(n/X)
 V(\frac{4\pi
\sqrt{nl}}{c_1})  W(\frac{4\pi \sqrt{nl'}}{c_2})\right\},
$$
where $\overline{x}$ is the multiplicative inverse of $x$ $(c_1)$
(resp. $\overline{y}$ for $y (c_2)).$ This is allowed because the
support of $g$ is compact, and therefore the sum over $n$ is
finite.

As the term in brackets in \ref{sm} is a smooth function, we can
apply Poisson summation to the $n$-sum to get,

\begin{equation}\label{eq:poisson} \lim_{X \to \infty} \frac{1}{X} \sum_{c_1,c_2} \frac{1}{c_1 c_2}\sum_{x  (c_1)^{*}}
 \sum_{y  (c_2)^{*}} e(\frac{\overline{x}l}{c_1}+\frac{\overline{y}l'}{c_2}) \end{equation}
 $$\left\{
 \sum_m
 \int_{-\infty}^{\infty} e(t(\frac{x c_2 + y c_1}{c_1 c_2}) -tm)g(t/X)
 V(\frac{4\pi
\sqrt{tl}}{c_1})  W(\frac{4\pi \sqrt{tl'}}{c_2})dt\right\}$$

Change of variables $t\rightarrow X t$, gives

\begin{equation} \lim_{X \to \infty} \sum_{c_1,c_2} \frac{1}{c_1 c_2}\sum_{x (c_1)^{*}}
 \sum_{y  (c_2)^{*}} e(\frac{\overline{x}l}{c_1}+\frac{\overline{y}l'}{c_2})\end{equation}
 $$\left\{
 \sum_{m \in \Z}
 \int_{-\infty}^{\infty} e(\frac{Xt(c_2x + c_1y -mc_1c_2)}{c_1c_2})g(t)
 V(\frac{4\pi
\sqrt{Xtl}}{c_1})  W(\frac{4\pi \sqrt{Xtl'}}{c_2})dt\right\} .$$

As we have fixed $l$ and $l'$, we write
$$ I(n,c_1,c_2,X): = \int_{-\infty}^{\infty}
e(\frac{Xtn}{c_1c_2})g(t)
 V(\frac{4\pi
\sqrt{Xtl}}{c_1})  W(\frac{4\pi \sqrt{Xtl'}}{c_2})dt
$$

Then $(L)$ is equal to the limit as $X\to\infty$ of
\[
\sum_{c_1,c_2} \frac{1}{c_1 c_2} \sum_{x (c_1)^{*}}
 \sum_{y  (c_2)^{*}}
e(\frac{\overline{x}l}{c_1}+\frac{\overline{y}l}{c_2}) \sum_{m \in
\mathbf{Z}}
 I(c_2x+c_1y-c_1c_2m,c_1,c_2,X)
\]
Note that for   fixed $X$,  the sums over $c_1$ and $c_2$ sums are
finite. Let $X'(c_1,c_2,n)$ be the set of solutions $(x,y,m)$ of
the equation
$$
c_2x + c_1y -mc_1c_2 = n
$$
where $x$ and $y$ range over a fixed set of representatives of
$(\mathbf{Z}/c_1)^*$ and $(\mathbf{Z}/c_2)^*$, respectively, and
$m\in\mathbf{Z}$. Then $(L)$ is equal to the limit as $X\to\infty$
of
$$
 \sum_{n\in\mathbf{Z}} \sum_{c_1,c_2} \frac{1}{c_1 c_2}\sum_{(x,y,m)\in X'(c_1,c_2,n)}
 e(\frac{\overline{x}l}{c_1}+\frac{\overline{y}l'}{c_2})  I(n,c_1,c_2,X)
$$

Note   that
$$
c_2x + c_1y -mc_1c_2 = c_2(x-mc_1) +c_1y
$$
Therefore, there is a bijection between the set of triples
$(x,y,m)\in X'(c_1,c_2,n)$ and the set of equivalent classes of
pairs $(x',y')$ in $X(c_1,c_2,n)$ from Definition \ref{equiv}. Thus we may replace the sum
over $X'(c_1,c_2,n)$ with a sum over $X(c_1,c_2,n)$:
\[
 (L) = \lim_{X\to\infty}\sum_{n\in\mathbf{Z}} \sum_{c_1,c_2} \frac{1}{c_1 c_2}\sum_{(x,y)\in X(c_1,c_2,n)}
 e(\frac{\overline{x}l}{c_1}+\frac{\overline{y}l'}{c_2})  I(n,c_1,c_2,X)
\]
Finally, let
\begin{equation} A_{n,X}:=\sum_{c_1,c_2} \frac{1}{c_1 c_2}\sum_{(x,y)\in X(c_1,c_2,n)}
 e(\frac{\overline{x}l}{c_1}+\frac{\overline{y}l'}{c_2})  I(n,c_1,c_2,X)\end{equation}
Then
 \begin{equation} (L) = \lim_{X \to
 \infty} \sum_{n \in \mathbf{Z}} A_{n,X}\end{equation}

Now define the standard Ramanajuan sum as $$f_n(m):=\sum_{s(n)^{*}}e(\frac{sm}{n}).$$

For $n=0$ using Lemma \ref{lem:n00} we have
$\overline{x}=-\overline{y}$ and \begin{equation}\label{a0x}
A_{0,X}=\sum_{c_1}
\frac{f_{c_1}(l-l')}{c_1^2}I(0,c_1,c_2,X)\end{equation}

Now for $n \neq 0,$ we can use the bijection of Proposition
\ref{bijection} to rewrite $A_{n,X}$ as a sum over  $r\in
Y(c_1,c_2,n)$:

\[
A_{n,X}=\sum_{r\in(\Z/n)^*} e(\frac{rl+\overline{r}l'}{n})
\sum_{\begin{tabular}{c}$c_1,c_2$\\$r\in
Y(c_1,c_2,n)$\end{tabular}} \frac{1}{c_1 c_2}
 e(\frac{lc_2}{c_1n}+\frac{l'c_1}{c_2n})  I(n,c_1,c_2,X)
 \]

 \begin{definition}
 Let $X(r)$ be the set of pairs $(c_1,c_2)$ such that $r\in Y(c_1,c_2,n)$.
 \end{definition}
\begin{definition}
Let
\begin{equation}\label{fn} F_{n}(x,y):=
 \frac{1}{xy}
e\left(\frac{1}{n}\left(\frac{ly}{x}+\frac{l'x}{y}\right)\right)
\times
\end{equation}

$$\times \left\{
 \int_{-\infty}^{\infty} e(\frac{tn}{xy})g(t)
 V(\frac{4\pi
\sqrt{tl}}{x})  W(\frac{4\pi \sqrt{tl'}}{y})dt\right\} .
$$

\end{definition}

 Now assuming \begin{equation}\label{switch} (L) = \lim_{X \to
 \infty} \sum_{n \in \mathbf{Z}} A_{n,X}=\sum_{n \in \mathbf{Z}} \lim_{X \to
 \infty}A_{n,X}, \end{equation}
the main result of the calculations can be broken down into the
cases: $n=0,$ and $n\neq 0.$

In Section \ref{ann} we show \begin{equation}\label{a0c}
\lim_{X \to
 \infty}A_{0,X}= \frac{6\delta_{l,l'}}{\pi^2}
\int_0^\infty V(y)W(y)\frac{dy}{y},
\end{equation}
where $\delta_{l,l'}$ is the Kronecker delta function.

While for $n\neq 0,$ and for all $r\in(\Z/n)^*,$
\begin{equation}\label{maincalc}
 \lim_{X\to\infty}\sum_{(c_1,c_2)\in X(r)} \frac{1}{c_1 c_2}
 e(\frac{lc_2}{c_1n}+\frac{l'c_1}{c_2n})  I(n,c_1,c_2,X) = \frac{6}{\pi^2}\frac1n\int_{0}^{\infty}
\int_{0}^{\infty} F_n(x,y)dxdy
\end{equation}

Summing this result for $r\in(\Z/n)^*,$ we get $$A_{n,X}=
\frac{6}{\pi^2} \frac{S(l,l',n)}{n}\int_{0}^{\infty}
\int_{0}^{\infty} F_n(x,y)dxdy.$$

The results from Section \ref{ann} then show
$$ (L)= \frac{6}{\pi^2}  \left(\delta_{l,l'}\int_0^\infty
V(y)W(y)\frac{dy}{y} +\sum_{n=1}^\infty
 \frac{S(l,l',n)}{n}\int_{0}^{\infty}
\int_{0}^{\infty} F_n(x,y)dxdy\right).$$

In Section \ref{sears}, $(L)$ is shown to be the geometric side of
a Kuznetsov trace formula. Taking the spectral side of this trace
formula completes Theorem \ref{main theorem1}[i.]. Reducing this to Rankin-Selberg orthogonality for individual cusp forms then occupies Sections $9,10,$ and $11.$

\section{Calculation for $A_{n,X}$}\label{ann}

\subsection{Case 1: $n\neq 0$}

Now fix $r,$ then by Proposition \ref{bijection}, $X(r)$ is the
set of  $(c_1,c_2)$ such that, setting $d=(c_1,c_2)$, we have
 \begin{enumerate}
 \item  $\frac{c_1}{d}, \frac{c_2}{d}$ are both prime to $\frac{n}{d}.$

 \item $\frac{c_1r+c_2}{d}\equiv 0\pmod{\frac{n}d}$
 \item
 $\frac{c_1r+c_2}{d}\not\equiv 0\pmod{\frac{n}{d'}}$ if $d'$ is a proper divisor of $d$.
 \end{enumerate}

 Now for each divisor $d$ of $n$,  let $X(r,d)$ be the set of pairs $(c_1,c_2)$ in $X(r)$ such that $(c_1,c_2)=d$.
 We would like to prove that there is a constant $R(n,d)$ such that
 \begin{equation}\label{maincalc2}
 \lim_{X\to\infty}\sum_{(c_1,c_2)\in X(r,d)} \frac{1}{c_1 c_2}
 e(\frac{lc_2}{c_1n}+\frac{l'c_1}{c_2n})  I(n,c_1,c_2,X) = R(n,d)\frac{6}{\pi^2}\frac1n\int_{0}^{\infty}
\int_{0}^{\infty} F_n(x,y)dxdy
\end{equation}
 and $$\sum_{d|n}R(n,d)=1$$

 Let us describe $X(r,d)$ explicitly. If $(c_1,c_2)\in X(r,d)$, then
  \begin{equation}\label{congr}
  \frac{c_2}d = -\frac{c_1}d\,r + \lambda\frac{n}d
  \end{equation}

  \begin{lemma} Fix  $c_1$  such that $c_1/d$ is prime to $n/d$. Let $\lambda$ be a whole number and define
  $c_2$ by (\ref{congr}). Then $(c_1,c_2)\in X(r,d)$ if and only if $(\lambda,c_1)=1$.
  \end{lemma}

  \begin{proof}
We show first that $(\lambda,c_1/d)=1$ if and only if
$\frac{c_2}{d}$ is relatively prime to $\frac{n}{d}$ and
$\frac{c_1}{d}.$ Assume that  $(\lambda,c_1/d)=1$. If  $p$
divides both $c_1/d$ and $c_2/d$,  then (\ref{congr}) gives
$p|n/d$, which is a contradiction. And if $p$ divides $c_2/d$ and
$n/d$, then (\ref{congr}) gives $p|r(c_1/d)$. But $(r,n)=1$, so
this implies that $p$ divides $c_1/d$ -- again a contradiction.

On the other hand, if $q = (\lambda,c_1/d)>1$, then $q$ divides
$c_2/d$. In this case, $\frac{c_2}d$ is not relatively prime
$\frac{c_1}d.$

If $d=1$, the only requirement is $(\lambda,c_1/d)=1$, i.e.,
$(c_1,c_2)\in X(r,d)$ if and only if $(\lambda,c_1)=1$.

If $d>1$, we must also require that
\begin{equation}\label{repeat}
 \frac{c_1r+c_2}n\not\equiv 0 \pmod{p\frac{n}d}
 \end{equation}
 for all $p|d$. But
 $$
  \frac{c_1r+c_2}d = \lambda\frac{n}d
 $$
Therefore (\ref{repeat}) holds if and only if $\lambda\not\equiv
0\pmod{p}$ for all $p|d$. In other words, $\lambda$ must be
relatively prime to $d$ as well. But the two conditions
$(\lambda,c_1/d)=1$ and $(\lambda,d)=1$ are equivalent to
$(\lambda,c_1)=1$.

\end{proof}
Perhaps it is more convenient to replace the pair $(c_1,c_2)$ with
a pair $(dc_1,dc_2)$ where $c_1, c_2$ are relatively prime to each
other and to $n/d$. Then $X(r,d)$ is describe by pairs
$(c_1,\lambda)$ and the left-hand side of
  (\ref{maincalc2}) is equal to
 \begin{equation}\label{maincalc3}
 \lim_{X\to\infty}\sum_{c_1: (c_1,n/d)=1} \sum_{ (\lambda,dc_1)=1 } \frac{1}{d^2c_1 c_2}
 e(\frac{lc_2}{c_1n}+\frac{l'c_1}{c_2n})  I(n,dc_1,dc_2,X)
\end{equation}
where
$$
 c_2 = -c_1\,r + \lambda\frac{n}d
 $$

\begin{definition}
Let
$H_{n}(x,y):=\frac{1}{xy}e(\frac{xl'}{ny}+\frac{yl}{nx})I(n,x,y,1).$

\end{definition}

Then (\ref{maincalc3}) equals
$$\lim_{X\to\infty}\frac{1}{X} \sum_{c_1: (c_1,n/d)=1} \sum_{
(\lambda,dc_1)=1 } H_n(\frac{dc_1}{\sqrt{X}},
\frac{dc_2}{\sqrt{X}}).$$
 
We prove
 \begin{prop}\label{Calc}
There exists a $1/2< \sigma_0 < 1,$ such that
\begin{equation}\label{mc}\frac{1}{X} \sum_{c_1: (c_1,n/d)=1} \sum_{ (\lambda,dc_1)=1 }
H_n(\frac{dc_1}{\sqrt{X}},
\frac{dc_2}{\sqrt{X}})=R(n,d)\frac{6}{\pi^2}\frac1n\int_{0}^{\infty}
\int_{0}^{\infty} F_n(x,y)dxdy + O(\frac{1}{n^2
X^{(1-\sigma_0)/2}})
\end{equation}
 where $$
 c_2 = -c_1\,r + \lambda\frac{n}d.
 $$ and $$\sum_{d|n}R(n,d)=1$$

The implied constant is independent of $n$ and $X.$
\end{prop}

\begin{proof}
The LHS of \ref{mc} equals
\begin{equation}\label{mc1}\frac{1}{X} \sum_{c_1: (c_1,n/d)=1} \sum_{ (\lambda,dc_1)=1 }
H_n(\frac{dc_1}{\sqrt{X}}, \frac{d(-c_1r+\frac{\lambda
n}{d})}{\sqrt{X}})\end{equation}
 Now
  fix $c_1,$ and define $$G(m):=H_n(\frac{d
 c_1}{\sqrt{X}},m).$$ Then the condition $(\lambda,dc_1)=1,$ is
 equivalent to $\lambda=s+dc_1q,$ for $1\leq s < dc_1, (s,dc_1)=1,
 q \in \Z.$
 We now fix an $s,$ and perform Poisson summation for the sum over $q,$ $$\sum_{q \in  \Z } G(\frac{-dc_1r +ns +
 ndc_1q}{\sqrt{X}}).$$
We get $$ \sum_{m \in \Z} \int_{-\infty}^{\infty}G(\frac{-dc_1r
+ns + ndc_1t}{\sqrt{X}})e(-mt)dt.$$ With a change of variables we
are left with \begin{equation}   \frac{\sqrt{X}}{ndc_1 }\sum_{m
\in \Z} e\left(\frac{(ns-dc_1
r)m}{ndc_1}\right)\widehat{G}(\frac{m\sqrt{X}}{
ndc_1}),\end{equation} where $\widehat{G}$ is the Fourier
transform of $G.$ Here
$$\widehat{G}(\frac{m\sqrt{X}}{ndc_1})=  \int_{-\infty}^{\infty}
e\left(\frac{1}{n}\left(\frac{l\sqrt{X}y}{dc_1}+\frac{l'dc_1}{\sqrt{X}y}\right)\right)
\times$$
$$ \left\{
 \int_{-\infty}^{\infty} e(\frac{\sqrt{X}tn}{dc_1y})g(t)
 V(\frac{4\pi\sqrt{ltX}}{dc_1})  W(\frac{4\pi
\sqrt{l't}}{y})dt \right\} e(\frac{- m
\sqrt{X}y}{ndc_1})\frac{dy}{y}.$$

the sum over $s$ gives
$$\sum_{s(dc_1)^{*}}e\left(\frac{ms}{dc_1}\right)=\mu\left(\frac{dc_1}{(m,dc_1)}\right)\frac{\phi(dc_1)}{\phi(\frac{dc_1}{(m,dc_1)})}.$$
We denote $$f_l(n):=
\mu\left(\frac{l}{(l,n)}\right)\frac{\phi(l)}{\phi(\frac{l}{(l,n)})}.$$
See (\cite{IK}).
 Note if $m=0,$ we have $\phi(dc_1).$

Therefore, we have
\begin{equation}\label{big} \frac{1}{n\sqrt{X}} \sum_{c_1:
(c_1,n/d)=1}\frac{f_{dc_1}(m)}{dc_1}  \sum_{m \in \Z}
e\left(\frac{-rm}{n}\right)P_{m,n}(\frac{d
c_1}{\sqrt{X}}),\end{equation} where
$$P_{m,n}(x):=\frac{1}{x}\int_{-\infty}^{\infty}
e\left(\frac{1}{n}\left(\frac{ly}{x}+\frac{l'x}{y}\right)\right)
\times$$
$$ \left\{
 \int_{-\infty}^{\infty} e(\frac{tn}{xy})g(t)
 V(\frac{4\pi\sqrt{lt}}{x})  W(\frac{4\pi
\sqrt{l't}}{y})dt \right\}e(\frac{- m y}{nx}) \frac{dy}{y}.$$

We define the Mellin transform of a function $F$ as $$
\widetilde{F} (s)= \int_0^{\infty}
F(x)x^{s}\frac{dx}{x}.
$$
 Since $P_{m,n}$ is smooth of compact support, integration by parts $M$
 times implies
 \begin{equation}\label{eq:Dec}\widetilde{P}_{m,n}(s)=O_M\left(\left(\frac{n}{(m(1+|t|))}\right)^{M}\right),\end{equation} where $s=\sigma +it.$ We now
 interchange the $c_1$ and $m$ sum. This is ok because the $c_1$ sum is finite. We now fix $m$ and study   \begin{equation}   \frac{1}{n\sqrt{X}} e(\frac{-rm}{n})\sum_{c_1:
(c_1,n/d)=1}\frac{f_{dc_1}(m)}{dc_1} \frac{1}{2\pi
i}\int_{Re(s)=\sigma_1}\widetilde{P}_{m,n}(s)(\frac{\sqrt{X}}{dc_1})^{s}ds
,\end{equation} where $\sigma_1$ is taken large enough to ensure
convergence.

As the $c_1$ sum is finite, we can interchange it and the integral
to get
\begin{equation} \label{LS}  \frac{1}{2\pi
ni\sqrt{X}}e(\frac{-rm}{n})\int_{Re(s)=\sigma_1}\widetilde{P}_{m,n}(s)L(s)(\sqrt{X})^{s}ds
,\end{equation} where $$L(s)=\sum_{c_1: (c_1,n/d)=1}
\frac{f_{dc_1}(m)}{(dc_1)^{s+1}}.$$

We have 2 parts: $m=0,$ and $m\neq 0.$

\subsubsection{Part 1} For $m=0,$ \begin{equation}L(s)=\left(\sum_{\substack{d'=\prod_{p} p^j \\
p|d,p\nmid
\frac{n}{d}}}\frac{\phi(d'd)}{(d'd)^{1+s}}\right)\left(\sum_{c_1:(c_1,n)=1}
\frac{\phi(c_1)}{c_1^{1+s}}\right).\end{equation}

For simplicity, define $$Z(d,s):=\left(\sum_{\substack{d'=\prod_{p} p^j \\
p|d,p\nmid \frac{n}{d}}}\frac{\phi(d'd)}{(d'd)^{1+s}}\right).$$

Then $$\label{Lf} L(s)=Z(d,s)\frac{L(s,\chi_0)}{L(s+1,\chi_0)}
$$
where $\chi_0$ is the trivial Dirichlet character modulo $n.$ It
has a pole at $s=1.$

Now we shift the contour in (\ref{LS}) from
$Re(s)=\sigma_1\rightarrow 3/4.$ The pole at $s=1$ has residue
$\frac{12}{\pi} \prod_{p|n} \frac{1}{(1+\frac{1}{p})}Z(d,1),$ and
rewrite (\ref{LS}) in the case of $m=0$ as
$$ \frac{6}{n\pi^2} \prod_{p|n}
\frac{1}{(1+\frac1{p})}Z(d,1)\widetilde{P}_{0,n}(1) +
\frac{1}{2\pi ni\sqrt{X}}
\int_{Re(s)=3/4}\widetilde{P}_{0,n}(s)L(s)(\sqrt{X})^{s}ds.$$

Now remember
\begin{equation}\label{fn} F_{n}(x,y):=
 \frac{1}{xy}
e\left(\frac{1}{n}\left(\frac{ly}{x}+\frac{l'x}{y}\right)\right)
\times
\end{equation}

$$\times \left\{
 \int_{-\infty}^{\infty} e(\frac{tn}{xy})g(t)
 V(\frac{4\pi
\sqrt{tl}}{x})  W(\frac{4\pi \sqrt{tl'}}{y})dt\right\} .
$$

 Integration by parts $k-$times in
(\ref{fn}) gives \begin{equation}\label{eq:Zerr}
F_n(x,y)=O_{x,y}\Bigl(\frac{1}{n^k}\Bigl).\end{equation}

Using trivial bounds on the integral and the bound \eqref{eq:Zerr}, we have
\begin{equation}\label{eq:alfn} \frac{6}{n\pi^2} \prod_{p|n}
\frac{1}{(1+\frac1{p})}Z(d,1)\widetilde{P}_{0,n}(1) +
O_n(\frac{1}{n^2 X^{1/8}}).\end{equation}

\subsubsection{Part 2} For $m \neq 0,$ the arguments are similar, but the
L-function equals \begin{equation}L(s)=\left(\sum_{\substack{d'=\prod_{p} p^j \\
p|d,p\nmid
\frac{n}{d}}}\frac{\phi(d'd)}{(d'd)^{1+s}}\right)\left(\sum_{c_1:(c_1,n)=1}
\frac{f_{c_1}(m)}{c_1^{1+s}}\right).\end{equation} As everything is multiplicative, we can
rewrite it as $Z(d,s)M(s),$ where

$$M(s)=\frac{1}{\zeta(s+1)} \sum_{\ell |n} \frac{\mu^2(\ell)}{\ell^{1+s}} \sum_{d |m} \frac{\mu(d)}{d^{1+s}}.$$ 

Now $Z(d,s)$ is entire, and  $M(s)$ is analytic for $\Re(s)>0.$
There exists $\sigma_0 < 1$ such that $\zeta(1+\sigma_0 +it) \neq 0$ for all $t \in \R.$ 
We shift the contour of the
integral to $\Re(s)=\sigma_0$ and get
\begin{equation}\label{eq:mm}\frac{1}{2\pi ni\sqrt{X}}
\int_{\Re(s)=\sigma_0}\widetilde{P}_{m,n}(s)L(s)(\sqrt{X})^{s}ds\end{equation}

To bound \eqref{eq:mm}, we use the bounds \eqref{eq:Dec} and \eqref{eq:Zerr}. Specifically, we can choose $M=2$ for \eqref{eq:Dec} and $k=M+2=4$ for \eqref{eq:Zerr}. This gives the bound  \begin{equation}\label{eq:Dac} \frac{1}{2\pi ni\sqrt{X}}
\int_{\Re(s)=\sigma_0}\widetilde{P}_{m,n}(s)L(s)(\sqrt{X})^{s}ds=O(\frac{1}{(nm)^2X^{(1-\sigma_0)/2}}).\end{equation}

Now for both cases $m=0$ and $m\neq0$, we have estimates \eqref{eq:alfn} and \eqref{eq:Dac}, respectively, to get (\ref{big}) equaling
\begin{equation} \frac{6}{n\pi^2} \prod_{p|n}
\frac{1}{(1+\frac1{p})}Z(d,1)\widetilde{P}_{0,n}(1) +
O(\frac{1}{n^2 X^{(1-\sigma_0)/2}}),\end{equation} after executing the $m$-sum.

Finally, notice $\widetilde{P}_{0,n}$ is $F_{n},$ and we have
\begin{equation}\label{fin}\frac{6}{n\pi^2} \prod_{p|n}
\frac{1}{(1+\frac1{p})}Z(d,1)\int_{0}^{\infty} \int_{0}^{\infty}
F_{n}(x,y)dxdy + O(\frac{1}{n^2 X^{(1-\sigma_0)/2}}).\end{equation}

Now let $$R(n,d):=Z(d,1)\prod_{p|n} \frac{1}{(1+\frac1{p})}.$$

\begin{lemma}
$\sum_{d \mid n} R(n,d) = 1.$
\end{lemma}

\begin{proof}
It suffices to do this for $n=p^l,$ $p$ a prime, $l \in \N.$ We
note in this case $d=p^i, 0 \leq i \leq l.$ Here
$$R(p^l,1)=\frac{1}{(1+\frac1{p})},$$ and
$$R(p^l,p^j)=\frac{1}{(1+\frac1{p})}\left(
\frac{\phi(p^j)}{p^{2j}}\right),$$ for $1<j<l.$ Lastly,
$$R(p^l,p^l)=\frac{1}{(1+\frac1{p})}\left(\sum_{k=0}^{\infty}
\frac{\phi(p^{l+k})}{p^{2l+2k}}\right).$$ Thus we only have to
prove
 \begin{equation}\label{pp} 1+ \left(\sum_{j=1}^{l-1}
\frac{\phi(p^j)}{p^{2j}}\right) + \left(\sum_{k=0}^{\infty}
\frac{\phi(p^{l+k})}{p^{2l+2k}}\right)=1+\frac1{p}.\end{equation}

For the middle sum of (\ref{pp}), we get $$\left(\sum_{j=1}^{l-1}
\frac{\phi(p^j)}{p^{2j}}\right)=(1-\frac{1}{p})\sum_{j=1}^{l-1}
\frac{1}{p^j}=\frac{p^{l-1}-1}{p^l}.$$

For the last sum of (\ref{pp}). we have
$$\left(\sum_{k=0}^{\infty}
\frac{\phi(p^{l+k})}{p^{2l+2k}}\right)=\frac{(1-\frac{1}{p})}{p^l}\sum_{k=0}^\infty
\frac{1}{p^k}=\frac{(1-\frac{1}{p})}{p^l(1-\frac{1}{p})}=\frac{1}{p^l}.$$

Summing the 3 terms then gives $$1+\frac{p^{l-1}-1}{p^l} +
\frac{1}{p^l}=1+\frac{1}{p}.$$

\end{proof}
This completes Proposition \ref{Calc}.

\end{proof}

\subsection{Case 2: $n=0$}

From (\ref{a0x}) we have,
\begin{equation}  \sum_{c_1}
\frac{f_{c_1}(l-l')}{c_1^2}I(0,c_1,c_2,X)=\end{equation}
$$= \sum_{c_1} \frac{f_{c_1}(l-l')}{c_1^2} \left\{
\int_{-\infty}^{\infty} g(t)
 V(\frac{4\pi
\sqrt{Xlt}}{c_1})  W(\frac{4\pi \sqrt{Xl't}}{c_1})dt\right\} .$$

\begin{prop}\label{RS_0}

\begin{equation}\label{eq:ia}
 \sum_{c_1} \frac{f_{c_1}(l-l')}{c_1^2} \left\{
\int_{-\infty}^{\infty} g(t)
 V(\frac{4\pi
\sqrt{Xlt}}{c_1})  W(\frac{4\pi \sqrt{Xl't}}{c_1})dt\right\} =
\end{equation}

\begin{equation*}= \frac{6\delta_{l,l'}}{\pi^2} \int_0^{\infty}
\int_0^\infty V(y)W(y)\frac{dy}{y}+O(\frac{1}{X^{3/4}}).
\end{equation*}
The implied constant is independent of $X.$
\end{prop}

\begin{proof}
 We define
\begin{equation} F(x):=\frac{1}{x}\int_{-\infty}^{\infty}
g(t)V(\frac{4\pi \sqrt{lt}}{x})W(\frac{4\pi
\sqrt{l't}}{x})dt.\end{equation} Denoting again the Mellin transform
of $F(x)$ as $ \widetilde{F}(s)$, and using the estimate
\eqref{eq:Dec}, we use Mellin inversion to write \eqref{eq:ia} as
\begin{equation}\label{eq:ja}
 \frac{1}{\sqrt{X}} \sum_{c_1}
\frac{f_{c_1}(l-l')}{c_1} F(\frac{c_1}{\sqrt{X}})=
  \frac{1}{\sqrt{X}}  \left\{ \frac{1}{2\pi
i}\int_{\sigma-i\infty}^{\sigma+i\infty}
\widetilde{F}(s)L(s)(\sqrt{X})^{s} ds \right\} ,\end{equation}
where 
\begin{equation}\label{eq:ramm} L(s)=\sum_{c=1}^\infty
\frac{f_{c}(l-l')}{c^{s+1}},\end{equation} and $\sigma$ is sufficiently large to ensure convergence of the
integral.

Now using the fact that $$f_n(m)=\sum_{r|(m,n)}\mu(\frac{n}{r})r,$$ we rewrite \eqref{eq:ramm} as 
\begin{equation}\label{eq:romm}
L(s)=\frac{1}{\zeta(1+s)}\sum_{r|(l-l')} \frac{1}{r^s}.\end{equation} This is certainly analytic for $\Re(s)>0.$

If $l=l',$ then
$$L(s)=\frac{\zeta{(s)}}{\zeta{(s+1)}}.$$ Shifting contour of the integral to 
$\sigma=3/4,$  $L(s)$ has a simple pole at only $s=1,$ only if $l=l'$ with residue
$\frac{6}{\pi^2}.$ After the shift, \eqref{eq:ja} equals
\begin{equation}   \frac{6\delta_{l,l'}}{\pi^2}
\widetilde{F}(1)+ O(\frac{1}{X^{3/4}}).
\end{equation}
 With a change of variables $y\rightarrow
\frac{4\pi \sqrt{tl}}{ y},$ we get \begin{equation} \frac{6\delta_{l,l'}}{\pi^2}
\int_0^\infty \int_0^\infty g(t)V(y)W(y)dt\frac{dy}{y}.
\end{equation} Using the fact $\int_0^\infty g(t)dt=1,$ we are
left with \begin{equation} \frac{6\delta_{l,l'}}{\pi^2} \int_0^\infty
V(y)W(y)\frac{dy}{y}+ O(\frac{1}{X^{3/4}}).
\end{equation}
\end{proof}

We now show the $n$-sum and limit can be interchanged.

\begin{lemma}
$\lim_{X \to \infty} \sum_{n \in \Z} A_{n,X}=\sum_{n \in \Z}
\lim_{X \to \infty} A_{n,X}.$
\end{lemma}

 \begin{proof}
We show $A_{n,X}$ is uniformly convergent in $X.$ Fix any
$\epsilon
>0,$ by Proposition \ref{Calc}, we have $$\Bigl|A_{n,X}-A_{m,X}\Bigl| =
\Bigl|\frac{C}{X^{(1-\sigma_0)/2}} \sum_{m}^n \frac{1}{n^2}\Bigl|
\leq \Bigl|\frac{C}{X^{(1-\sigma_0)/2}} \int_{m}^n \frac
{dx}{(x+1)^2}\Bigl|,$$ where $C$ is a fixed constant independent
of $n$ and $X$ and $1/2 < \sigma_0 <1.$ Suppose $n \geq m \neq 0$
then,
$$\Bigl|\frac{C}{X^{(1-\sigma_0)/2}} \int_{m}^n \frac {dx}{(x+1)^2}\Bigl| \leq
\Bigl|\frac{C}{X^{(1-\sigma_0)/2}}\frac{2}{m+1}\Bigl|.$$ Since we
only take $X$ in the range $[1,\infty),$ and $(1-\sigma_0)/2 > 0,$
we have uniform convergence in $X$ by taking $n,m \geq
M(\epsilon),$ such that $M(\epsilon):=\frac{2C}{\epsilon} -1.$
Thus the sum and limit can be interchanged.

\end{proof}

\section{Analysis of $\int_0^\infty \int_0^\infty F_n (x,y)dx
dy$}

We extend the integrals from $(-\infty,\infty)$ so we can write
this as
\begin{equation} \int_{-\infty}^{\infty} \int_{-\infty}^{\infty}
  e\left(\frac{1}{n}(
 \frac{l'x}{y}+\frac{ly}{x})\right) \times
\end{equation}

$$ \times \left\{
 \int_{-\infty}^{\infty} e\left( \frac{tn}{x  y} \right)g(t)
 V(\frac{4\pi
\sqrt{lt}}{x})  W(\frac{4\pi \sqrt{l't}}{ y})dt
\right\}\frac{dx}{x} \frac{dy}{y} .$$

After a change in variables $x\rightarrow x\sqrt{tl}, y\rightarrow
y\sqrt{tl'},$ we get \begin{equation}\label{eq:bes1}
\int_{-\infty}^{\infty} \int_{-\infty}^{\infty}
F_{n,l,l'}(x,y)dxdy= \int_{-\infty}^{\infty}
\int_{-\infty}^{\infty}
  e\left(\frac{\sqrt{ll'}}{n}(
 \frac{x}{y}+\frac{y}{x})\right) e\left( \frac{n}{\sqrt{ll'}x  y}\right)\times\end{equation} $$V(\frac{4\pi
}{x})  W(\frac{4\pi}{ y})\frac{dx}{x} \frac{dy}{y} \left\{
 \int_{-\infty}^{\infty} g(t) dt \right\}
.$$

Let \begin{equation} F(z):= \int_{-\infty}^{\infty}
\int_{-\infty}^{\infty}
  \exp \left(z\frac{i}{2}(
 \frac{x}{y}+\frac{y}{x})\right) \exp \left( (\frac{1}{z})\frac{8\pi^2i}{xy}\right) V(\frac{4\pi
}{x})  W(\frac{4\pi}{ y})\frac{dx}{x} \frac{dy}{y}
\end{equation} and $G(z):=F(z)+F(-z).$ So \eqref{eq:bes1} equals
$F(\frac{4\pi \sqrt{ll'}}{n}).$ We include $F(-z)$ for in \eqref{eq:poisson} the sum is over the integers. The analysis in the previous sections is identical for $n$ or $-n,$ but this integral must be accounted for in the final calculation.   

\subsection{Computation for $J$-Bessel function.}

Remembering that the J-Bessel
transform is
\begin{equation} h(V,k)=i^k \int_0^{\infty} V(x)
J_{k-1}(x)\frac{dx}{x}.
\end{equation} 

\begin{prop}\label{besseltrans}
Let $k$ be an even integer, then \begin{equation}h(G,k) = 2\pi h(V,k)
\cdot h(W,k)\end{equation}
\end{prop}

\begin{proof}
It is sufficient to study this for $F(z).$ We note first \begin{equation} h(F,k)= i^{k}
\int_0^\infty F(w) J_{k-1}(w)\frac{dw}{w}= \int_{-\infty}^{\infty}
\int_{-\infty}^{\infty}V(\frac{4\pi }{x}) W(\frac{4\pi}{ y})
\times
\end{equation}
$$ \times \Bigg(i^{k}\int_0^{\infty}
  \exp \left(w\frac{i}{2}(
 \frac{x}{y}+\frac{y}{x})\right) \exp\left( (\frac{1}{w})\frac{8\pi^2i}{xy}\right)J_
 {k-1}(w)\frac{dw}{w}\Bigg) \frac{dx}{x}
 \frac{dy}{y}.$$
Now make a change of variables $x\rightarrow
\frac{4\pi}{x},y\rightarrow \frac{4\pi}{y},$ to get $$
h(F,k)=
\int_{-\infty}^{\infty}\int_{-\infty}^{\infty}V(x) W(y) \times
$$
\begin{equation}\label{bes2} \times
\Bigg(i^{k}\int_0^{\infty}
  \exp \left(w\frac{i}{2}(
 \frac{x}{y}+\frac{y}{x})\right) \exp \left( \frac{ixy}{2w}\right)J_
 {k-1}(w)\frac{dw}{w}\Bigg) \frac{dx}{x}
 \frac{dy}{y}.\end{equation}

Notice the test functions $V$ and $W$ are chosen to be supported on the positive real numbers. We study the integral in the $w$ variable in (\ref{bes2}). First
we make a change of variables $w\rightarrow \frac{xy}{-i w},$
yielding

\begin{equation}i^{k}\int_0^{-i\infty}
  \exp \left(-(
 \frac{x^2+y^2}{2w})\right) \exp \left( \frac{w}{2}\right)J_
 {k-1}(\frac{-iyx}{w})\frac{dw}{w}. \end{equation}

 The J-Bessel function transforms by $J_{k-1}(ix)=i^{k-1}I_{k-1}(x)$ and
 $J_{k-1}(-x)=-J_{k-1}(x),$ since k-1 is odd. Thus

\begin{equation}\label{eq:bes4}\frac{1}{i}\int_{-i\infty}^0
  \exp \left( \frac{w}{2}-(
 \frac{x^2+y^2}{2w}) \right)I_
 {k-1}(\frac{xy}{w})\frac{dw}{w}. \end{equation}

Now doing the same analysis for $F(-z),$ we obtain

\begin{equation}\label{eq:bes5}\frac{1}{i}\int_0^{i\infty}
  \exp\left( \frac{w}{2}-(
 \frac{x^2+y^2}{2w}) \right)I_
 {k-1}(\frac{xy}{w})\frac{dw}{w}. \end{equation}

 Adding \eqref{eq:bes4} and \eqref{eq:bes5} we get
\begin{equation}\label{eq:bes6} \frac{1}{i}\int_{-i\infty}^{i\infty}
  \exp\left( \frac{w}{2}-(
 \frac{x^2+y^2}{2w}) \right)I_
 {k-1}(\frac{xy}{w})\frac{dw}{w}. \end{equation}

We now state a formula from \cite{Wat},

\begin{equation}\label{eq:Wat}
J_{\nu}(Z)J_{\nu}(y) = \frac{1}{2\pi i} \int_{-i \infty}^{i
\infty} \exp \left( t/2 -
(\frac{Z^2+y^2}{2t})\right)I_{\nu}(\frac{yZ}{t})\frac{dt}{t}.
\end{equation}

 Using \eqref{eq:Wat}, \eqref{eq:bes6} equals

\begin{equation} 2 \pi J_{k-1}(x) J_{k-1}(y). \end{equation}

Incorporating \eqref{eq:bes6} into $h(G,k)$ we get
\begin{equation} \int_{-\infty}^{\infty} \int_{-\infty}^{\infty}
V(x)W(y)  \left( 2 \pi J_{k-1}(x) J_{k-1}(y)\right)
\frac{dx}{x}\frac{dy}{y}=
\end{equation}
$$=2\pi i^k\int_{-\infty}^{\infty} V(x)J_{k-1}(x)\frac{dx}{x}
i^k\int_{-\infty}^{\infty}W(y)J_{k-1}(y)\frac{dy}{y}=2\pi
h(V,k)\cdot h(W,k).
$$

\end{proof}

We now must show $F(z)$ is also the convolution for the $B$-Bessel function.

\subsection{Computation for $B$-Bessel function.}

Again, $B_{2it}(x):=(2\sin(\pi i t))^{-1}(J_{-2it}(x)-J_{2it}(x)),$ and $h(V,t):=\int_0^\infty V(x)B_{2it}(x)x^{-1}dx, t \in \mathbb{R}, V \in C_0^\infty(\mathbb{R}).$ 
    \begin{prop}\label{Btrans}  Let $G(z):= F(z) + F(-z),$ and $t$ purely imaginary, then $h(G,t)= \pi h(V,t)h(W,t).$ \end{prop}    
 \begin{proof}
The goal is study $F(z),$ similar calculations can be done for $F(-z).$
 We note first \begin{equation} h(F,t)= 
\int_0^\infty F(w) B_{2it}(w)\frac{dw}{w}= \int_{-\infty}^{\infty}
\int_{-\infty}^{\infty}V(\frac{4\pi }{x}) W(\frac{4\pi}{ y})
\times
\end{equation}
$$ \times \Bigg(\int_0^{\infty}
  \exp \left(w\frac{i}{2}(
 \frac{x}{y}+\frac{y}{x})\right) \exp\left( (\frac{1}{w})\frac{8\pi^2i}{xy}\right)B_
 {2it}(w)\frac{dw}{w}\Bigg) \frac{dx}{x}
 \frac{dy}{y}.$$

 Now make a change of variables $x\rightarrow
\frac{4\pi}{x},y\rightarrow \frac{4\pi}{y},$ to get $$
h(F,t)=
\int_{-\infty}^{\infty}\int_{-\infty}^{\infty}V(x) W(y) \times
$$
\begin{equation}\label{eq:bes22} \times
\Bigg(\int_0^{\infty}
  \exp \left(w\frac{i}{2}(
 \frac{x}{y}+\frac{y}{x})\right) \exp \left( \frac{ixy}{2w}\right)B_
 {2it}(w)\frac{dw}{w}\Bigg) \frac{dx}{x}
 \frac{dy}{y}.\end{equation}
 
 Using that the $B$-Bessel function is a difference of imaginary order $J$-Bessel functions it is sufficient to focus on 
 
 $$\frac{1}{2\sin(\pi i t)}\int_{-\infty}^{\infty}\int_{-\infty}^{\infty}V(x) W(y) \times
$$
\begin{equation}\label{eq:bes22} \times
\Bigg(\int_0^{\infty}
  \exp \left(w\frac{i}{2}(
 \frac{x}{y}+\frac{y}{x})\right) \exp \left( \frac{ixy}{2w}\right)J_
 {-2it}(w)\frac{dw}{w}\Bigg) \frac{dx}{x}
 \frac{dy}{y}.\end{equation}
The integral in $J_{2it}$ will be a similar calculation. We study the integral in the $w$ variable in (\\eqref{eq:bes22}. First
 a change of variables $w\rightarrow \frac{xy}{-i w}$ is made,
yielding

\begin{equation}\label{eq:bess}T_{F}^{-}(x,y):=\int_0^{-i\infty}
  \exp \left(-(
 \frac{x^2+y^2}{2w})\right) \exp \left( \frac{w}{2}\right)J_
 {-2it}(\frac{-iyx}{w})\frac{dw}{w}. \end{equation}

 The J-Bessel function of imaginary order transforms by $J_{-2it}(\pm ix)=e^{\pm \pi t} I_{-2it}(x)$ and $J_{2it}(\pm ix)=e^{\mp \pi t} I_{2it}(x)$  by inspection of the power series.  Thus $T_{F}^{-}(x,y)$ equals

\begin{equation}\label{eq:bes44}e^{-\pi t}\int_0^{-i\infty}
  \exp \left( \frac{w}{2}-(
 \frac{x^2+y^2}{2w}) \right)I_
 {-2it}(\frac{xy}{w})\frac{dw}{w}. \end{equation}

  From \cite{Wat}(chap 13.7), we now borrow two formulas

\begin{equation}\label{eq:Wat2}
H_{\nu}^{(1)}(Z)J_{\nu}(y) = \frac{1}{\pi i} \int_{0}^{c+i
\infty} \exp \left( t/2 -
(\frac{Z^2+y^2}{2t})\right)I_{\nu}(\frac{yZ}{t})\frac{dt}{t},
\end{equation}

\begin{equation}\label{eq:Wat3}
H_{\nu}^{(2)}(Z)J_{\nu}(y) = \frac{-1}{\pi i} \int_{0}^{c-i
\infty} \exp \left( t/2 -
(\frac{Z^2+y^2}{2t})\right)I_{\nu}(\frac{yZ}{t})\frac{dt}{t}.
\end{equation}
 
 Here $H^{(i)}$ is the $i$-th order Hankel function.
 Therefore, $T_{F}^{-}(x,y)$ equals $ -\pi i e^{-\pi t}H_{-2it}^{(2)}(x)J_{-2it}(y).$ 
 Likewise, the term from \eqref{eq:bes22} with Bessel transform $J_{2it},$ which we will call $T_{F}^{+}(x,y) $ is  $ \pi i e^{\pi t}H_{2it}^{(2)}(x)J_{2it}(y).$ 
 
 This gives 
 
 \begin{equation}\label{firstbes} (T_{F}^{-}+T_F^{+})(x,y)= \pi i \big(-e^{-\pi t}H_{-2it}^{(2)}(x)J_{-2it}(y) + e^{\pi t}H_{2it}^{(2)}(x)J_{2it}(y)\big).\end{equation}
 
 Remember the aim of the proposition is for the function $G(z):= F(z) + F(-z).$ Similar calculations are now done for $F(-z).$ The calculations up to \eqref{eq:bess} are the same except we make the change of variables $w\rightarrow \frac{xy}{i w}$  here giving
\begin{equation}\label{eq:bess11}T_{F(-z)}^{-}(x,y):=\int_0^{i\infty}
  \exp \left(-(
 \frac{x^2+y^2}{2w})\right) \exp \left( \frac{w}{2}\right)J_
 {-2it}(\frac{iyx}{w})\frac{dw}{w}. \end{equation} 
 
 By a similar use of equations \eqref{eq:Wat2}, \eqref{eq:Wat3}, one obtains for \eqref{eq:bess11} $\pi i e^{\pi t}H_{-2it}^{(1)}(x)J_{-2it}(y).$ For the  $J_{2it}$ transform, which we label $T_{F(-z)}^{+}(x,y)$ one gets analogously $-\pi i e^{-\pi t}H_{2it}^{(1)}(x)J_{2it}(y).$  
 
 Thus, $$W(x,y):=T_{F}^{-}(x,y)+T_{F}^{+}(x,y)+T_{F(-z)}^{-}(x,y)+T_{F(-z)}^{+}(x,y)=$$
 \begin{equation}\label{eq:gcomp}= \frac{1}{2\sin(\pi i t)} \pi i \Big(-e^{-\pi t}\big(H_{-2it}^{(2)}(x)J_{-2it}(y)+H_{2it}^{(1)}(x)J_{2it}(y)\big) + e^{\pi t}\big(H_{-2it}^{(1)}(x)J_{-2it}(y)+H_{2it}^{(2)}(x)J_{2it}(y)\big) \Big).\end{equation}

 $H_{\alpha}^{(i)}(x)$ can be expanded into $J-$Bessel functions as:
 \begin{equation}\label{han1}
H_{\alpha}^{(1)}(x)= \frac{J_{-\alpha}(x)-e^{(-\alpha \pi i)}J_{\alpha}(x)}{i \sin(\alpha \pi)},
\end{equation}
and
 \begin{equation}\label{han2}
H_{\alpha}^{(2)}(x)= \frac{J_{-\alpha}(x)-e^{(\alpha \pi i)}J_{\alpha}(x)}{-i \sin(\alpha \pi)}.
\end{equation}
 
 Then expanding the LHS of \eqref{eq:gcomp} using these identities we have, \begin{equation}
\frac{-\pi  e^{-\pi t}}{2\sin(\pi i t) \sin(2\pi i t)}\Bigg[\Big(J_{2it}(x)J_{-2it}(y)-e^{2\pi t}J_{-2it}(x)J_{-2it}(y)\Big)+\end{equation} 
$$\Big(J_{-2it}(x)J_{2it}(y)-e^{2\pi t}J_{2it}(x)J_{2it}(y)\Big)\Bigg]
$$

The RHS of \eqref{eq:gcomp} is 
 \begin{equation}
\frac{-\pi  e^{\pi t}}{ 2\sin(\pi i t)\sin(2\pi i t)}\Bigg[\Big(J_{-2it}(x)J_{2it}(y)-e^{2\pi t}J_{2it}(x)J_{2it}(y)\Big)+ \Big(J_{2it}(x)J_{-2it}(y)-e^{2\pi t}J_{-2it}(x)J_{-2it}(y)\Big)\Bigg]
\end{equation}

Regathering terms,  $W(x,y)$ equals $$ \frac{-\pi}{2\sin(\pi i t)\sin(2\pi i t)}\Bigg(J_{2it}(x)J_{2it}(y)[e^{-\pi t} +e^{\pi t}] +J_{-2it}(x)J_{2it}(y)[-e^{-\pi t} -e^{\pi t}]  $$

\begin{equation}+J_{2it}(x)J_{-2it}(y)[-e^{-\pi t} - e^{\pi t}] + J_{-2it}(x)J_{-2it}(y)[e^{-\pi t} + e^{\pi t}] \Bigg).\end{equation}

Using $\sin(2\pi i t)=2\cos(\pi i t)\sin(\pi i t)$ and $\cos(\pi i t)=\frac{e^{-\pi t} + e^{\pi t}}{2}$ and regathering terms again, 
\begin{equation}\label{eq:besfinal}
\frac{\pi}{2\sin(\pi i t)\sin(2\pi i t)}[e^{-\pi t} + e^{\pi t}][J_{-2it}(x)-J_{2it}(x)][ J_{-2it}(y)-J_{2it}(y)]=
\end{equation}
$$=\pi B_{2it}(x)B_{2it}(y).$$

Incorporating \eqref{eq:besfinal} into $h(G,t)$ we get
\begin{equation} \int_{-\infty}^{\infty} \int_{-\infty}^{\infty}
V(x)W(y)  \left(  \pi B_{2it}(x)B_{2it}(y)\right)
\frac{dx}{x}\frac{dy}{y}=
\end{equation}
$$=\pi \int_{-\infty}^{\infty} V(x) B_{2it}(x)\frac{dx}{x}
\int_{-\infty}^{\infty}W(y) B_{2it}(y)\frac{dy}{y}=\pi
h(V,t)h(W,t).
$$

 \end{proof}  

This proves Theorem \ref{main theorem}, and for now we define $V*W(z):=G(z).$

\subsection{Sears-Titchmarsh Inversion}\label{sears}

\begin{definition} Let $f\in L^{2}(\R^{+},\frac{dx}{x}),$ then
\begin{equation} f(x) = 4\int_0^{\infty} h(f,t) \tanh(\pi t) B_{2it}(x) tdt +
2\sum_{k>0,k\text{ even}} (k-1)J_{k-1}(x)h(f,k),\end{equation}
where $ h(f,t):=\int_0^{\infty} f(x)B_{2it}(x)\frac{dx}{x} $ and
$h(f,k) := i^{k}\int_0^{\infty} f(x)J_{k-1}(x)\frac{dx}{x}.$ Further if
$$f^{\infty}(x): = 4\int_0^{\infty}  h(f,t) \tanh(\pi t)
B_{2it}(x) tdt$$ and
$$f^{0}(x):= 2\sum_{k>0,k\text{ even}} (k-1)J_{k-1}(x) h(f,k),$$ then
$f(x)=f^{0}(x) + f^{\infty}(x).$ \end{definition} This is the
Sears-Titchmarsh inversion formula. See \cite{Iw} for reference.

\begin{prop}\label{pr5} Let $M(t) = h(V,t)h(W,t)$ then,\begin{align}
  \int_0^{\infty} V(x)W(x)
\frac{dx}{x} =  \left[\int_0^{\infty} V(x)W^{\infty}(x)
\frac{dx}{x} \right]+ \left[\int_0^{\infty} V(x)W^{0}(x)
\frac{dx}{x} \right] &= \\ 2\left(\left[\int_{-\infty}^\infty
M(t)\tanh(\pi t) t dt \right] + \left[ \sum_{k>0,k\text{ even}}
(k-1)M(k) \right]\right).
\end{align}

\end{prop}
\begin{proof}

Expressing $B_{2it}(x)$ as a difference of J-Bessel functions, it
is easy to see it is an even function in the variable $t$.
Exploiting this, we see by a change of variables, $$
2\int_{-\infty}^\infty M(t)\text{tanh}(\pi t) t dt=
4\int_{0}^\infty M(t)\text{tanh}(\pi t) t dt.$$ Expanding $M(t)$
 we have \begin{equation}4\int_{0}^\infty \left(\int_0^{\infty} V(x)B_{2it}(x)\frac{dx}{x} \right)
 h(W,t)\text{tanh}(\pi t) t
 dt.\end{equation} Now since $V$ has compact support we can and do interchange the integrals, \begin{equation}4\int_{0}^\infty V(x)\left(\int_0^{\infty} B_{2it}(x)
 h(W,t)\text{tanh}(\pi t) tdt\right) \frac{dx}{x}
 .\end{equation} By Sears-Titchmarsh inversion, this equals
 \begin{equation}\int_0^{\infty}V(x)W^{\infty}(x)\frac{dx}{x}
 .\end{equation}

 We now focus on
$2\sum_{2k>0,k\in \N} (k-1)M(k).$ Expanding $M(k)$ again, we get
\begin{equation} 2\sum_{k>0,k \text{ even}} (k-1)
\left(\int_0^{\infty} V(x)J_{k-1}(x)\frac{dx}{x} \right)
h(V,k).\end{equation} Interchanging the sum and the integral gets
\begin{equation} \int_0^{\infty} V(x) \left(2\sum_{k>0,k\text{ even}} (k-1)h(W,k)J_{k-1}(x) \right)\frac{dx}{x}= \int_0^{\infty}
V(x)W^{0}(x)\frac{dx}{x}.
\end{equation} Summing the parts from the B-Bessel and J-Bessel functions, we get our proposition
\begin{equation}
=\left[2\int_{-\infty}^\infty M(t)\text{tanh}(\pi t) t dt \right]
+ \left[2 \sum_{k>0,k\text{ even}} (k-1)M(k) \right]
 = \int_0^{\infty}
V(x)W(x)\frac{dx}{x}. \end{equation}

\end{proof}

From Section \ref{ann} we have shown \begin{equation}\label{gf} L=
\frac{6}{\pi^2}  \left(\delta_{l,l'}\int_0^\infty V(y)W(y)\frac{dy}{y} +
\sum_{n=1}^\infty \frac{S(l,l',n)}{n}(V*W)(\frac{4\pi
\sqrt{ll'}}{n})\right).\end{equation}

We now use the Sears-Titchmarsh inversion formula for $V*W(z)$ getting,
\begin{equation}\label{gg}  V*W(z) = 4\pi \left(\int_0^{\infty} M(t) \text{tanh}(\pi t) B_{2it}(z) tdt +
\sum_{k>0,k\text{ even}} (k-1)J_{k-1}(z)M(k)\right),\end{equation}
where as before
\begin{equation*}M(t)= h(V,t)h(W,t).\end{equation*} Remember the $\pi$ factor comes from Theorem \ref{main theorem}.

 While also using Proposition \ref{pr5} and equation
(\ref{gg}), we can write (\ref{gf}) as
\begin{equation}\label{yoy} \left[\frac{12}{\pi^2}\int_{-\infty}^\infty M(t)\text{tanh}(\pi t)
t dt \right] + \left[\frac{12}{\pi^2} \sum_{k>0,k\text{ even}}
(k-1)M(k) \right] + \end{equation} $$ +
\frac{48}{\pi}\left(\sum_{c=1}^\infty \frac{S(l,l',c)}{c}
\left(\int_0^{\infty} M(t) \text{tanh}(\pi t)
B_{2it}(\frac{4\pi\sqrt{ll'}}{c}) tdt + \sum_{k>0,k\text{ even}}
(k-1)J_{k-1}(\frac{4\pi\sqrt{ll'}}{c})M(k)\right)\right).$$

\begin{theorem}(Kuznetsov trace formula) Denote the Maass form of eigenvalue $1/4 + t^2$ by $\phi_t,$ and let $\eta(l,1/2+it):=2\pi^{1+it}{\cosh(\pi t)}^{-1/2} \frac{\tau_{it}(n)}{\Gamma(1/2+it)\zeta(1+2it)},$ 
where $\tau_{it}(n)=\sum_{ab=n}(a/b)^{it}.$ Then 

\begin{equation}   \sum_{\phi_t } G(t_{\phi})
{a_l(\phi_t)}{\overline{a_l'(\phi_t)}}  + \frac{1}{4
\pi}\int_{-\infty}^\infty G(t)\eta(l,1/2+it)\overline{\eta(l',1/2+it)}
dt=\end{equation}

\begin{equation*}  = \delta_{l,l'}G_0 +
\sum_{c=1}\frac{S(l,l',c)}{c}G^{+}(4\pi
\sqrt{ll'}/c),\end{equation*} where $G_0:=
\frac{1}{\pi}\int_{-\infty}^\infty G(t)\tanh(\pi t) t dt,$ and
$G^{+}(x): = 4\int_0^{\infty} G(t) \tanh(\pi t) B_{2it}(x) tdt.$
\end{theorem}

\begin{theorem}(Petersson trace
formula) Let the holomorphic forms of weight $k$ be denoted as $\phi_k,$  then

\begin{equation} \sum_{k>0,k \text{ even}} \sum_{\phi_k } G(k_{\phi}){a_l(\phi_k)}{\overline{a_l'(\phi_k)}}
=  \frac{1}{\pi}\sum_{k>0,k \text{ even}} (k-1)\delta_{l,l'}G(k) +
\sum_{c=1}^{\infty} \frac{S(l,l',c)}{c} \widehat{G}(4\pi
\sqrt{ll'}/c),
\end{equation} where
\begin{equation} \widehat{G}(x) = 4\sum_{k>0,k \text{ even}}
(k-1)G(k)J_{k-1}(x). \end{equation}
\end{theorem}

See \cite{Iw} for more details of these two trace formulas.

Incorporating these trace formulas into (\ref{yoy}), we get
(\ref{gf}) equals

\begin{equation}   \frac{12}{\pi}\left( \sum_{\phi_t } M(t_{\phi})
{a_l(\phi_t)}{\overline{a_l'(\phi_t)}}  + \frac{1}{4
\pi}\int_{-\infty}^\infty M(t)\eta(l,1/2+it)\overline{\eta(l',1/2+it)} dt +
\sum_{k>0,k \text{ even}} \sum_{\phi_k }
M(k_{\phi}){a_l(\phi_k)}{a_l'(\phi_k)}\right).\end{equation} This
proves Theorem \ref{main theorem1}.

\section{Matching for the continuous spectrum}
We prove Theorems \ref{cupcon1}, \ref{cupcon2}, and \ref{cupcon3} in this section. For Rankin-Selberg orthogonality one needs to match cuspidal terms with cuspidal terms, and continuous terms with continuous terms, i.e. showing  \begin{equation}\label{eq:matchcc} \lim_{X \to \infty} \frac{1}{X} \sum_{n \in \Z}
g(n/X) S_{n,l}(V)  S_{n,l'}(W)=\frac{12}{\pi}S_{l,l'}(V*W),
\end{equation}
and
\begin{equation}\label{eq:matchcc1} \lim_{X \to \infty} \frac{1}{X} \sum_{n \in \Z}
g(n/X) C_{n,l}(V)  C_{n,l'}(W)=\frac{12}{\pi}C_{l,l'}(V*W).
\end{equation}

We must also show cuspidal terms must be orthogonal to the continuous terms, or 

 \begin{equation}\label{matchcuscon}\lim_{X \to \infty} \frac{1}{X} \sum_{n \in \Z}
g(n/X) S_{n,l}(V)  C_{n,l'}(W)=0.
\end{equation}

We prove these propositions here.

\begin{prop}
\begin{equation}\label{matchconc}\lim_{X \to \infty} \frac{1}{X} \sum_{n \in \Z}
g(n/X) C_{n,l}(V)  C_{n,l'}(W)=\frac{12}{\pi}C_{l,l'}(V*W).
\end{equation}

\end{prop}

\begin{proof}
Our claim fully written out is
$$\lim_{X \to \infty} \frac{1}{X} \sum_{n} g(n/X)\big[\frac{1}{4\pi}\int_{-\infty}^{\infty}h(V,T)\eta(n,1/2+iT)\overline{\eta(l,1/2+iT)}dT\big]\times$$
$$ \big[\frac{1}{4\pi}\int_{-\infty}^{\infty}h(W,T))\eta(n,1/2+it)\overline{\eta(l',1/2+it)dt}\big]=
\frac{3}{\pi^2}\int_{-\infty}^{\infty}h(V,z)h(W,z)\eta(l,1/2+iz)\overline{\eta(l',1/2+iz)}dz$$

Assuming the interchanging of sums and using the functional equation for the gamma function: $$\Gamma(1/2+it)\Gamma(1/2-it)=\frac{\pi}{\cosh{\pi t}},$$ this boils down to studying 
\begin{equation}
 \int_{T} \frac{h(V,T)\tau_{iT}(l)}{|\zeta(1-2iT)|^2} \int_{t} \frac{h(W,t)\tau_{it}(l')}{|\zeta(1-2it)|^2} \times 
\end{equation}

$$\times \frac{1}{2 \pi i} \int_{\sigma=4} \hat{g}(s)[\Pi_{\pm,\pm}\zeta(s\pm iT \pm it)]\frac{X^{s}}{\zeta(2s)}ds.dt.dT.$$ 
The last equation follows from mellin inversion and Ramanujan's formula.

Now doing a contour shift from $\sigma \to 1/2,$ we pick up poles at $1\pm iT \pm it,$.  The left over integral is $O_{T,t}(X^{1/2}),$ and is negligible. The term to compute then is

\begin{equation}\label{rz}
\lim_{X \to \infty} \frac{1}{2 \pi i}\int_{T} h(V,T)\tau_{iT}(l) \int_{t}h(W,t) \tau_{it}(l') \Big[\frac{X^{-iT}}{\zeta(1+2iT)}\Big(\frac{X^{-it}\hat{g}(1-iT-it)\zeta(1-2iT-2it)}{\zeta(1+2it)\zeta(2-2iT-2it)} \Big) +\end{equation}

$$+ \Big(\frac{X^{it}\hat{g}(1-iT+it)\zeta(1-2iT+2it)}{\zeta(1-2it)\zeta(2-2iT+2it)} \Big) \Big]+
$$

$$+\Big[\frac{X^{iT}}{\zeta(1-2iT)}\Big(\frac{X^{-it}\hat{g}(1+iT-it)\zeta(1+2iT-2it)}{\zeta(1+2it)\zeta(2+2iT-2it)} \Big) +$$

$$+ \Big(\frac{X^{it}\hat{g}(1+iT+it)\zeta(1+2iT+2it)}{\zeta(1-2it)\zeta(2+2iT+2it)} \Big) \Big]dtdT.
$$
The term $X$ has been factored out of the residue calculation, so (\ref{rz}) should be $O(1)$ after taking the limit. 
It is sufficient to study the first of these four integrals. We make a change of variables $T \to T-t$ to get \begin{equation}
\lim_{X \to \infty}  \frac{1}{2 \pi i}\int_{t}h(W,t) \frac{\tau_{it}(l')}{\zeta(1+2it)}  \int_{-\infty}^{\infty} \frac{h(V,T-t)\tau_{i(T-t)}(l) X^{-iT} \hat{g}(1-iT)\zeta(1-2iT)dT}{\zeta(2-2iT)  \zeta(1+2i(T-t))}
\end{equation}
 
Here $\zeta(1-2iT)$ has a pole at $T=0,$ and to understand this we use the following lemma. 

\begin{lemma}
Let $H$ be a differentiable function in $L^{1}(\mathbb{R}),$ then

\medskip
$PV \int_{-\infty}^{\infty} H(x)e^{ikx}\frac{dx}{x}:= \lim_{\epsilon \to 0^{+}} \int_{|x| \geq \epsilon}H(x)e^{ikx}\frac{dx}{x} \to \pm \pi i H(0)$ as $\pm k \to \infty.$ 
\end{lemma}
\begin{proof}
See \cite{Venk1}, Lemma 10. 
\end{proof}

 Applying this lemma for $k=-\log X,$ we obtain: \begin{equation}\label{pvcs}
\hat{g}(1)\frac{1}{2\zeta(2)}
\int_{-\infty}^{\infty} \frac{h(V,-t)h(W,t) \tau_{it}(l')\tau_{-it}(l)}{\zeta(1+2it)  \zeta(1-2it)}dt
\end{equation}
Now $\hat{g}(1)=1$ by definition, and $h(V,-t)=h(V,t),$ because $B_{2it}(x)$ is real valued. Likewise,  $\tau_{-it}(l)=\tau_{it}(l).$

In recovering $\eta(l,1/2+it)$ from $\tau_{it}(l),$ (\ref{pvcs}) becomes $$ \frac{3}{4\pi^{2}}
\int_{-\infty}^{\infty} h(V,t)h(W,t) \eta(l,1/2+it)\overline{\eta(l',1/2+it)}dt.$$  Summing then over the four integrals in (\ref{rz}) completes our proposition and Theorem \ref{cupcon2}.

\end{proof}

\begin{prop} \begin{equation}\label{matchcuscon}\lim_{X \to \infty} \frac{1}{X} \sum_{n \in \Z}
g(n/X) S_{n,l}(V)  C_{n,l'}(W)=0.
\end{equation}
 \end{prop}

\begin{proof}
Using Mellin inversion (\ref{matchcuscon}) is written as \begin{equation}\lim_{X \to \infty} \frac{1}{X} \frac{1}{4\pi}\sum_{\phi}
h(V,\lambda_{\phi}) \int_{-\infty}^\infty h(V,t)  \big[ \frac{1}{2 \pi i} \int_{\sigma=4} \hat{g}(s) L(s)X^{s}ds\big]dt,\end{equation} where $$L(s,t)=\sum_{n=1}^{\infty} \frac{a_n(\phi)\eta(l,1/2+it)}{n^s}.$$ Now up to some analytically harmless factors, which come from normalizations from the trace formula, $$L(s,t)\approx \frac{L(\phi,s+it)L(\phi,s-it)}{\zeta(2s)}.$$ This has no pole at $s=1,$ and thus we can do a contour shift in the $\hat{g}$ integral from $4 \to 3/4.$ The integral in the $s$ variable is certainly bounded and the limit is $$\lim_{X \to \infty} O_{t,\phi}(X^{-1/4})=0.$$ This completes Theorem \ref{cupcon3}.

\end{proof}

Incorporating these propositions into Theorem \ref{main theorem1} gives Theorem \ref{cupcon1}.

\section{Reduction to a single archimedean parameter}

In the last section we showed \begin{equation}\label{eq:matchcc} \lim_{X \to \infty} \frac{1}{X} \sum_{n \in \Z}
g(n/X) S_{n,l}(V)  S_{n,l'}(W)=S_{l,l'}(V*W).
\end{equation}

 We can reduce \eqref{eq:matchcc} from an infinite spectral sum equality to an equality of cusp forms of the same weight or eigenvalue parameter. This is done in this section using the fact that $(L)$ holds for a large class of test functions $V,W$ and their associated transforms $h(V,t),h(W,t).$ This reduces the problem to a "finite dimensional" matching problem. The argument we use is summarized in 2 propositions in the appendix of \cite{Venk1}.

\begin{remark}We mean "finite dimensional" in the sense that for a given even positive integer $k,$ there are finitely many cusp forms of weight $k.$ Likewise, we expect the space of Maass forms of eigenvalue $1/4 +t_j^2$ to be one dimensional.  \end{remark}

\begin{prop}\label{venkm1}
Let $t_j$ be a discrete subset of $\mathbb{R}$ with $\{j: t_j \leq T\} \ll T^{r}$ for some r. Let, for each j, there be given a function $c_{X}(t_j)$ depending on $X$, so that $c_{X}(t_j) \ll t_{j}^{r'}$ for some $r'$- the implicit constant independent of $X$; similarly, for each $k$ odd, let there be given a function $c_{X}(k)$ depending on $X$ so that $c_{X}(k) \ll k^{r'}.$ Suppose that one has an equality 
\begin{equation}
\lim_{X \to \infty} \big(\sum_{j} c_{X}(t_j)h(V,t_j)+\sum_{k \text{ odd}} c_{X}(k)h(V,k)\big)=0
\end{equation}
for all $(h(V,t_j),h(V,k))$ that correspond via Sears-titchmarsh inversion to $V$. Then $\lim_{X \to \infty} c_{X}(t_j)$ exists for each $t_j$ and equals 0, and similarly the same holds for $\lim_{X \to \infty}c_{X}(k).$ This equality holds for all functions $h$ for which both sides converge.
\end{prop}  

\begin{prop}\label{venkm2}
Given $j_0 \in \mathbb{N}, \epsilon > 0$ and an integer $N >0,$ there is a $V$ of compact support so that $h(V,t_j)=1,$ and for all $j' \neq j_0, h(V,t_{j'}) \ll \epsilon(1+|t_{j'}|)^{-N},$ and for all $k$ odd, $h(V,k) \ll \epsilon k^{-N}.$ 

Given $k_0, \epsilon >0$ and an integer $N>0,$ there is a $V$ of compact support so that $h(v,k_0)=1, h(V,k) \ll \epsilon k ^{-N}$ for $k$ odd $k\neq k_0,$ and $h(V,t) \ll (1+|t|)^{-N}$ for all $\mathbb{R}.$ 
\end{prop}

Using Propositions \ref{venkm1} and \ref{venkm2}, we can choose our test functions $V,W$ such that their associated Bessel transforms are  supported on weights $k$ or eigenvalue parameters $t_j.$ Upon expanding the right hand side of  \eqref{eq:matchcc}, \begin{equation}
\frac{12}{\pi}\sum_{\phi }h(V,t_{\phi})h(W,t_{\phi})a_{l}(\phi)a_{l'}(\phi)
\end{equation} one sees 
 that only choosing both the test functions to be supported on the same weight or eigenvalue will have an associated non-zero contribution. Certainly this agrees with Rankin-Selberg theory. Choose now $V,W$ to be supported on an eigenvalue parameter $t_j,$ say, as in Proposition \ref{venkm2}. Then \eqref{eq:matchcc} reduces to 

\begin{equation}\label{eq:finite}
\lim_{X \to \infty} \frac{1}{X} \sum_{n \in \Z}
g(n/X)  \big( \sum_{\substack{\phi_t \\ t=t_j} }a_{n}(\phi)\overline{a_{l}(\phi)}\big)\big(\sum_{\substack{\psi_t \\ t=t_j} }a_{n}(\psi)\overline{a_{l'}(\psi)}\big) =
\end{equation}
 
 $$=\frac{12}{\pi}\sum_{\substack{\phi_t \\ t=t_j} }a_{l}(\phi)\overline{a_{l'}(\phi)}.
$$
   Here as in Proposition \ref{venkm2}, we choose the transforms such that $h(V,t)=1$ for $t=t_j.$
   
  We would like to interchange the limit and the spectral sum, but this requires knowing that the limit  $$\lim_{X \to \infty} \frac{1}{X} \sum_n g(n/X) a_n(\phi)a_n(\psi)$$ exists. If we assume Rankin-Selberg orthogonality then we certainly get this. However the point of the beyond endoscopy approach is to not make such assumptions. 
  
  What one needs to interchange the limit and spectral sum is to build in Hecke operators into our trace formula. This and the analytic continuation of the Rankin-Selberg L-function we show in a following paper.

   \end{document}